\documentclass{article}
\usepackage[utf8]{inputenc}
\usepackage{geometry}
\usepackage{natbib}
\usepackage{mathtools}
\usepackage{physics}
\usepackage{amsthm}
\usepackage{amsmath}
\usepackage{graphicx} % Required for inserting images
\usepackage{amsmath}
\usepackage{amssymb}
\usepackage{pdflscape}
\usepackage{graphicx} % Required for inserting images
\usepackage{enumerate}
\usepackage{caption, subcaption}
\usepackage{color}

\newcommand{\PP}{\mathbb{P}}

\newtheorem{lemma}{Lemma}
\newtheorem{theorem}{Theorem}

\title{Limiting distributions of ratios of Binomial random variables}
\author{Adriel Barretto\footnote{University of Virginia, Department of Statistics, Charlottesville, VA, USA.}\phantom{a}and Zachary Lubberts\footnotemark[1]}
\date{}

\begin{document}

\maketitle

\begin{abstract}
We consider the limiting distribution of the quantity $X^s/(X+Y)^r$, where $X$ and $Y$ are two independent Binomial random variables with a common success probability and a number of trials $n$ and $m$, respectively, and $r,s$ are positive real numbers. Under several settings, we prove that this converges to a Normal distribution with a given mean and variance, and demonstrate these theoretical results through simulations. 
\end{abstract}

\section{Introduction}
As part of the analysis in \cite{lubberts2025}, the authors consider the quantity 
$$ \frac{X}{\sqrt{X+Y}}, $$
where $X\sim \mathrm{Binomial}(n,p),$ $Y\sim\mathrm{Binomial}(m,p)$, and $X, Y$ are independent of one another. The setting of interest in that paper occurs when $m$ has order comparable to $n^2$, so that the quantity in the denominator is typically of a similar order to the one in the numerator, and in this case, it was shown that for large values of $n$, this ratio  has a limiting Normal distribution with mean $\sqrt{p}$ and variance $(1-p)/n$. While it is well-known that given the value of $X+Y$, for some random variables $X$ and $Y$ as above, the value of $X$ is a Hypergeometric random variable, the literature is less forthcoming about the distribution of such ratios of a pair of Binomial random variables. In the present work, we extend the previous results to the case where the function takes the form
$$
R =\frac{X^s}{(X+Y)^r},
$$
for $r,s>0$. We will show that under several parameter regimes, $R$ has a limiting Normal distribution. We verify these results through simulations, showing the effects of varying each of the parameters on the observed distribution. 

\section{Limiting Distribution}

We may approximate the function defining our ratio of interest using the Taylor approximation:
\begin{equation} f(x,y) = \frac{x^s}{(x+y)^r} = \frac{(np)^s}{(np+mp)^r}+\nabla f(np,mp)^T\begin{bmatrix} x-np\\ y-mp\end{bmatrix} +Q(x,y), \label{eq:approx}\end{equation}
where the quadratic remainder term takes the form 
$$Q(x,y) = \frac{1}{2}\begin{bmatrix} x-np & y-mp \end{bmatrix} \nabla^2 f(\xi, \eta) \begin{bmatrix}x-np\\y-mp\end{bmatrix}, $$
for some point $(\xi,\eta)$ on the line segment connecting $(np,mp)$ with $(x,y)$. 

If we substitute the random variables $X$ and $Y$ into this expression, the linear term in Equation~(\ref{eq:approx}) already looks quite Normal in distribution once $n$ and $m$ are large, since $X,Y$ are independent Binomial random variables. However, in order to show that the limiting distribution of $R$ looks like the linear terms in this equation, we must bound the remainder term. We will make use of the following result from \cite{chung2000course}:

\begin{lemma}
\label{lem:slutzky}
Let $X_n\rightarrow X$ in distribution, and $Y_n\rightarrow 0$ in probability. Then $X_n+Y_n\rightarrow X$ in distribution.
\end{lemma}

We will show that after appropriate scaling, the quadratic remainder term converges to 0 in probability, so the two linear terms determine the limiting distribution of $R$. For any quadratic form $x^TAx$, applying the Cauchy-Schwarz inequality and monotonicity property of the spectral norm, we have
$$ x^TAx = \langle Ax,x\rangle \leq \|Ax\|\|x\|\leq \|A\|\|x\|^2,$$ so in order to control $Q,$ we must find an upper bound for $\|\nabla^2 f(x,y)\|_2$. Note that a Hessian matrix must be symmetric, so its singular values are simply the absolute values of its eigenvalues. The Gerschgorin disk theorem \cite{horn2012matrix} tells us that any eigenvalue of a matrix must belong to the union of its Gerschgorin disks, so no eigenvalue of a matrix $A$ can be larger than 
$$\max_i \sum_{j=1}^n |A_{ij}|\leq \sum_{i,j=1}^n |A_{ij}|.$$ In the present case, the Hessian is given by 
$$\nabla^2 f(x,y) = \frac{x^{s-2}}{(x+y)^{r+2}}\begin{bmatrix}s(s-1)(x+y)^2-2rsx(x+y)+r(r+1)x^2& r(r+1)x^2-rsx(x+y)\\ r(r+1)x^2-rsx(x+y)& r(r+1)x^2\end{bmatrix}.$$ Utilizing the Gerschgorin bound just stated, we have
\begin{align}
\|\nabla^2 f(x,y)\|_2 &\leq \frac{x^{s-2}}{(x+y)^{r+2}}\left[|s(s-1)(x+y)^2-2rsx(x+y)+r(r+1)x^2|\right.\notag\\
&\quad \left.+2|r(r+1)x^2-rsx(x+y)|+r(r+1)x^2\right].\label{eq:hessianbound}
\end{align}

Now since $X,Y$ are Binomial random variables, we know that with overwhelming probability, $|X-np|\leq \sqrt{n\log(n)p(1-p)}$, and $|Y-mp|\leq \sqrt{m\log(m)p(1-p)}$; let us call this event $\mathcal{A}$. More precisely, the probability that $|X-np|>C\sqrt{n\log(n)}$ decays faster than $Cn^{-2}$, and similarly for $Y$. So to bound the residual term, we see that the event
$$ \{|Q(X,Y)|>\epsilon\}\subseteq \{|Q(X,Y)|>\epsilon\}\cup \{|X-np|>C\sqrt{n\log n}\}\cup \{|Y-mp|>C\sqrt{m\log m}\},$$
and thus 
\begin{align*}
\PP[|Q(X,Y)|>\epsilon]&\leq \PP[(|Q(X,Y)|>\epsilon) \cap (|X-np|\leq C\sqrt{n\log n})\cap (|Y-mp|\leq C\sqrt{m\log m})]\\
&\quad +C n^{-2} + Cm^{-2}.
\end{align*}

Consider $x=np+r_x$, where $|r_x|\leq C\sqrt{n\log(n)}$, and $y=mp+r_y$, where $|r_y|\leq C\sqrt{m\log(m)}$. Then the right hand side of inequality (\ref{eq:hessianbound}) may be bounded by
$$
\frac{Cn^{s-2}\max\{s(n+m),rn\}^2}{(n+m)^{r+2}}.$$ We also have the inequality
\begin{align*}
\left\|\begin{bmatrix}x-np\\y-mp\end{bmatrix}\right\|^2&= r_x^2+r_y^2\leq Cn\log(n)+Cm\log(m),
\end{align*}
so on the event $\mathcal{A}$, we get that
\begin{align}
|Q(X,Y)| &\leq \frac{Cn^{s-2}\max\{s(n+m),rn\}^2}{(n+m)^{r+2}}(n\log(n)+m\log(m))\notag\\
&\leq \frac{Cn^{s-2}\log(n+m)}{(n+m)^{r-1}},
\label{eq:qbound}
\end{align}
since $s,r$ are constants. The exact behavior of this quantity depends on the ratio $m/n$, but in light of Lemma~\ref{lem:slutzky}, in order to obtain the distributional convergence results, we must simply show that the final quantity in (\ref{eq:qbound}) still goes to zero after appropriate scaling, for any of the cases we wish to consider.

To find the limiting distribution of $R=f(X,Y)$, we re-arrange Equation~(\ref{eq:approx}) after evaluating at $(X,Y)$ (and neglecting the remainder term for the time being):
\begin{align*}
R - \frac{(np)^s}{((n+m)p)^r} &\approx \nabla f(np,mp)^T \begin{bmatrix}X-np\\Y-mp\end{bmatrix}\\
&= \frac{(np)^s}{((n+m)p)^{r+1}}\begin{bmatrix}\left(\frac{s(n+m)}{n}-r\right)\sqrt{np(1-p)}& -r\sqrt{mp(1-p)}\end{bmatrix}\begin{bmatrix}\frac{X-np}{\sqrt{np(1-p)}}\\ \frac{Y-mp}{\sqrt{mp(1-p)}}\end{bmatrix}.
\end{align*}
The last vector will converge to a Normal vector $(Z_X,Z_Y)$ as $n,m\rightarrow\infty$, but an appropriate scaling is required so that neither of the coefficients of $Z_X, Z_Y$ diverge as $n,m\rightarrow\infty$, and it is not the case that both of them vanish (since in that case, the limiting distribution is simply point mass at 0). The choice of scaling will again depend on the ratio $m/n$, but now we have all of the ingredients in place to prove the following theorem:

\begin{theorem}
\label{thm:main}
Let $r,s>0$, $p\in(0,1)$ be constants. Let $X\sim \mathrm{Binomial}(n,p)$, and let $Y\sim\mathrm{Binomial}(m,p)$ be independent of $X$. Define
$$ 
R= \frac{X^r}{(X+Y)^s}.
$$
Then as $n,m\rightarrow\infty$, we have the following convergences in distribution:
\begin{enumerate}[(i)]
\item If $m/n\rightarrow\infty$, $m\log(m)n^{-3/2}\rightarrow 0$, then
$$
\frac{m^r}{n^{s-1/2}}\left(R-\frac{n^s}{(n+m)^r}p^{s-r}\right)\rightarrow \mathcal{N}\left(0,p^{2(s-r)-1}(1-p)s^2\right).
$$
\item If $m/n\rightarrow\alpha\in(0,+\infty)$, then
$$
n^{r-s+1/2}\left(R-\frac{n^s}{(n+m)^r}p^{s-r}\right)\rightarrow \mathcal{N}\left(0,p^{2(s-r)-1}(1-p)\frac{[(s(1+\alpha)-r)^2+\alpha r^2]}{(1+\alpha)^{2(r+1)}} \right).
$$
\item If $m/n\rightarrow0$, then
$$
n^{r-s+1/2}\left(R-\frac{n^s}{(n+m)^r}p^{s-r}\right)\rightarrow \mathcal{N}\left(0,p^{2(s-r)-1}(1-p)(s-r)^2\right).
$$
\end{enumerate}
\end{theorem}

\begin{proof}
It remains to be shown that in the three cases described above, the bound in (\ref{eq:qbound}) still goes to zero after multiplication by the appropriate scaling factor as $n,m\rightarrow\infty$. 
\begin{enumerate}[(i)]
\item When $m/n\rightarrow\infty$, $m\log(m)n^{-3/2}\rightarrow 0$, then
$$ C \frac{n^{s-2}\log(n+m)}{(n+m)^{r-1}} \frac{m^r}{n^{s-1/2}}\sim \frac{m
\log(m)}{n^{3/2}}\rightarrow 0.$$
\item When $m/n\rightarrow\alpha$, then
$$ C \frac{n^{s-2}\log(n+m)}{(n+m)^{r-1}}n^{r-s+1/2}\sim \frac{\log(n)}{\sqrt{n}}\rightarrow 0.$$
\item When $m/n\rightarrow 0$, then
$$ C \frac{n^{s-2}\log(n+m)}{(n+m)^{r-1}} n^{r-s+1/2} \sim \frac{\log(n)}{\sqrt{n}}\rightarrow 0.$$
\end{enumerate}
\end{proof}

\section{Simulation}

We verify our limiting distribution results for different values of the parameters $n, m, p, r$, and $s$. For each set of parameter values we tested, we generate 100,000 points $(X_i,Y_i)$, where $X_i\sim \mathrm{Binomial}(n,p)$, $Y_i\sim \mathrm{Binomial}(m,p)$, and $X_i,Y_i$ are independent. We then compute $\frac{ X_i^s }{(X_i+Y_i)^r}$, and center and scale according to Theorem~\ref{thm:main}, obtaining $R_i$. We compare this sample with a sample of 100,000 values $Z_j$ drawn from the Normal distribution with mean 0 and the appropriate variance. To measure the accuracy between the distributions of $R$ and $Z$, we use the discrete KL divergence: We divide the values in $\{R_i\}, \{Z_j\}$ into 100 bins $x\in \mathbb{X}$, then set $A(x)$ to be the observed proportion of $R_i\in x$ and $B(x)$ to be the observed proportion of $Z_j\in x$. We compare the two distributions with the formula

% To find the simulated model, 100000 points were randomly sampled from binomial(m,p) and binomial(n,p) distributions. The function, $\frac{x^s}{(x+y)^r}$ was then executed on each points and $\frac{(np)^s}{(np+mp)^r}$
%  was subtracted to obtain a value, K. Then, K was multiplied
% by the appropriate factor (either $n^{r-s-\frac{1}{2}}$,$\frac{m^{r-1}}{n^{s-2}}$ or $\frac{m^{r-1}}{n^{s-\frac{1}{2}}}$ to obtain a simulated distribution, A. Then, a hypothetical distribution,$B$, was found by (i.e. the normal distributions) randomly generating 100000 points.
% To test the accuracy between the simulated model, A, and the hypothetical model,B, Discrete KL Divergence was used and defined as below"
\begin{equation}
D_{KL}(A||B) = \sum_{x \in \mathbb{X}} A(x)[\log(A(x))-\log(B(x))].
\end{equation}
When $\frac{m}{n}\rightarrow \infty$, it can happen that all of the observations fall into a single bin, so in this case we use the reversed formula to get a meaningful comparison of the distributions:
\begin{equation}
D_{KL}(B||A) = \sum_{x \in \mathbb{X}} B(x)[\log(B(x))-\log(A(x))].
\end{equation}

We test the distributions in each of the settings of Theorem~\ref{thm:main}, as well as the case where $m/n\rightarrow\infty$, but $m\log(m)/n^{3/2}\not\rightarrow 0$. Whenever we fix the value of $r$ or $s$, it is set equal to 15 (except in Section~\ref{s:m0}, for reasons that we explain there), and whenever we fix the value of $p$, it is set equal to 0.5. For each of the regimes of $m/n$, all plots show the effect of varying one of the parameters while the others remain fixed. 

% In addition, for each case, when $\frac{m}{n}\rightarrow \infty$, when $\frac{m}{n^{\frac{3}{2}}}\rightarrow 0$ and $\frac{m}{n}\rightarrow \infty$, when $\frac{m}{n}\rightarrow 0$, and when $\frac{m}{n}\rightarrow c$, where $c$ is some constant, m and n were picked such that the cases apply. r and s were both fixed at 15, and p was fixed at .5. For each case then, one of these parameters, $m,n,p,r,s$ was changing, while the other parameters remained fixed at their values. The KL Divergence was measured to see how varying the parameter impacted how accurate the hypothetical distribution vs against the simulated distribution. Below are the results.
\subsection{$\frac{m}{n}\rightarrow \infty, m\log(m)n^{-3/2}\not\rightarrow 0$}
\label{s:mi}
While outside the scope of the theorem, for this setting, whenever we fix $m$ or $n$, we choose $m= 2\times 10^9$, $n=2\times 10^5$. Whenever we fix one of $r,s$, we set its value to be $15$, and whenever we fix $p$, we set its value to be 0.5. The results are shown in Figure~\ref{fig1}. We observe that the KL divergence remains concentrated around 3.35 very consistently despite changing values of the parameters. This occurs since the limiting distribution ends up collapsing to point mass at 0, as a result of the denominator growing much faster than the numerator in $$ \frac{X^s}{(X+Y)^r}. $$

\begin{figure}
\centering
\begin{subfigure}{0.32\textwidth}
\includegraphics[width=\linewidth]{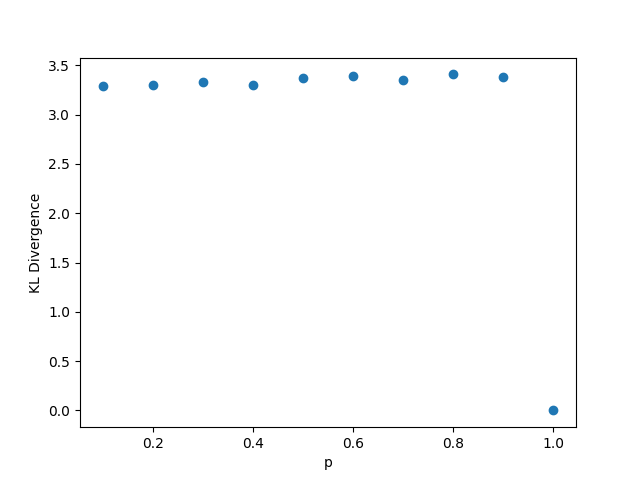}
\caption{Changing $p$ (range 0 to 1)}
\label{sfig:1a}
\end{subfigure}
\begin{subfigure}{0.32\textwidth}
\includegraphics[width=\linewidth]{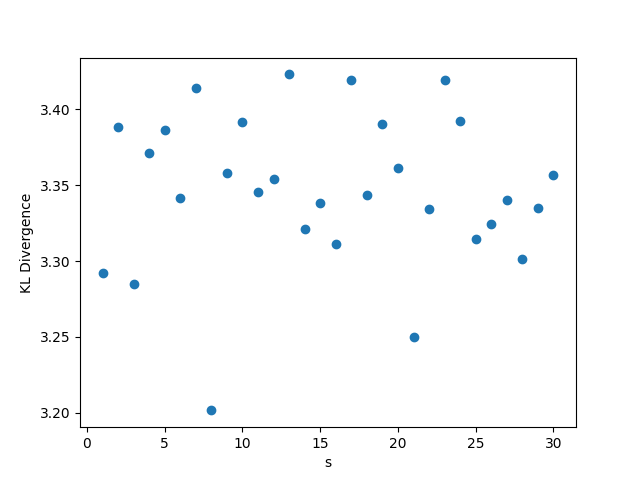}
\caption{Changing $s$ (range 1 to 30)}
\label{sfig:1b}
\end{subfigure}
\begin{subfigure}{0.32\textwidth}
\includegraphics[width=\linewidth]{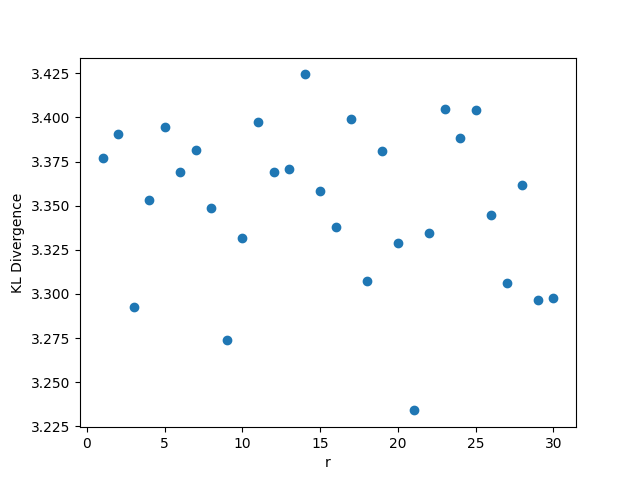}
\caption{Changing $r$ (range 1 to 30)}
\label{sfig:1c}
\end{subfigure}

\begin{subfigure}{0.33\textwidth}
\includegraphics[width=\linewidth]{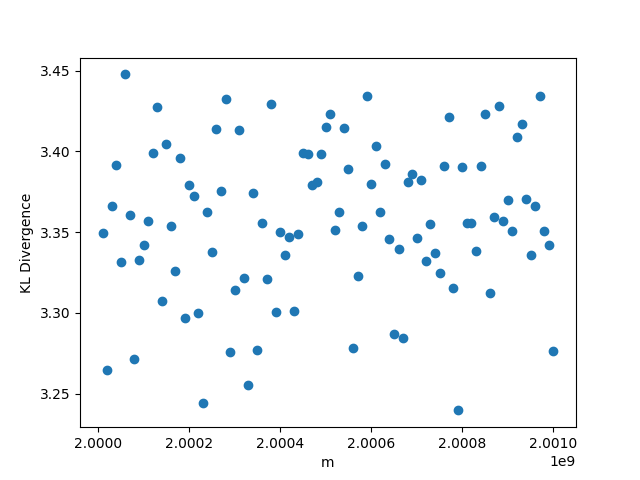}
\caption{Changing $m$ (range $2\times 10^9$ to $2.001\times 10^9$)}
\label{sfig:1d}
\end{subfigure}
\hspace{0.05\textwidth}
\begin{subfigure}{0.33\textwidth}
\includegraphics[width=\linewidth]{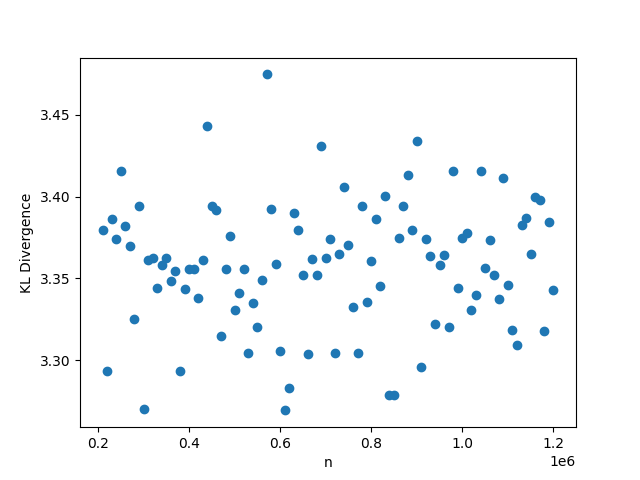}
\caption{Changing $n$ (range $2\times 10^5$ to $1.2\times 10^6$)}
\label{sfig:1e}
\end{subfigure}
\caption{Effect of changing parameters on the KL divergence comparing samples coming from our true distribution with those of the Normal distribution, in the setting where $m/n\rightarrow\infty$ and $m\log(m) n^{-3/2}\not\rightarrow 0$. See Section~\ref{s:mi} for more details.}
\label{fig1}
\end{figure}

\subsection{$m/n\rightarrow\infty, m\log(m)n^{-3/2}\rightarrow 0$}
\label{s:mf}
In this setting, whenever we fix $m$, we choose the value $1.1\times 10^9$, and $n=3.8\times 10^6$. As usual, when we fix the other parameters, we choose $r,s=15$ and $p=0.5$. The results may be seen in Figure~\ref{fig2}. In this setting, the KL divergence is nearly zero regardless of changes in the parameters. The low value of the KL divergence indicates that the simulated and hypothetical distributions are nearly identical, reinforcing Theorem~\ref{thm:main}. We note that when $r$ grows, the distribution of $X^r$ will become increasingly skewed, so $n$ and $m$ may need to be larger to get the same degree of convergence in distribution. While we can see in panel (c) that the KL divergence does increase with $r$, the value is still very small (0.012) even for $r=30$ for these values of $m$ and $n$, however.

\begin{figure}
\centering
\begin{subfigure}{0.32\textwidth}
\includegraphics[width=\linewidth]{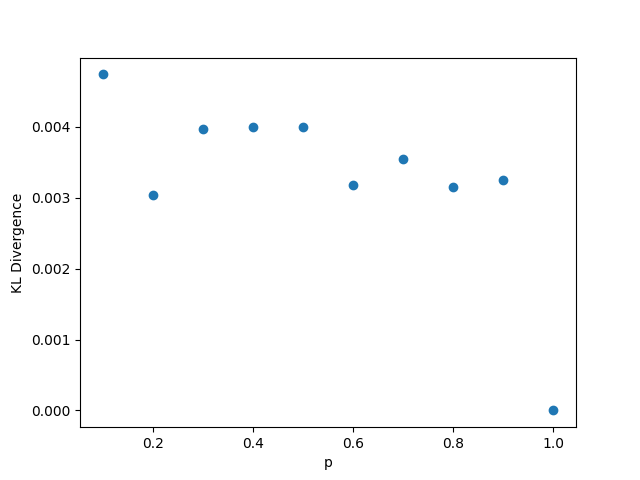}
\caption{Changing $p$ (range 0 to 1)}
\label{sfig:2a}
\end{subfigure}
\begin{subfigure}{0.32\textwidth}
\includegraphics[width=\linewidth]{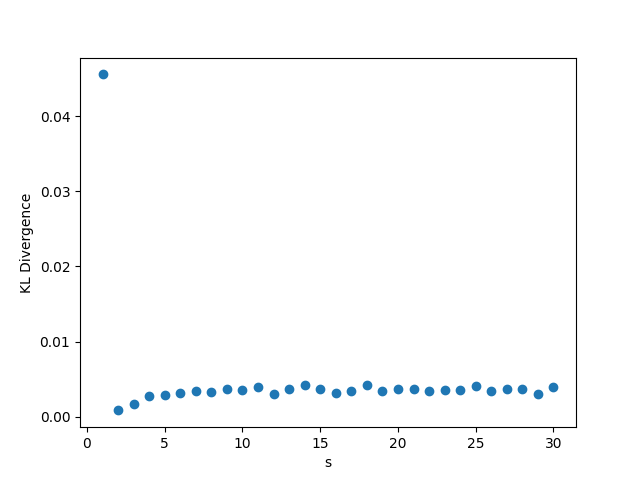}
\caption{Changing $s$ (range 1 to 30)}
\label{sfig:2b}
\end{subfigure}
\begin{subfigure}{0.32\textwidth}
\includegraphics[width=\linewidth]{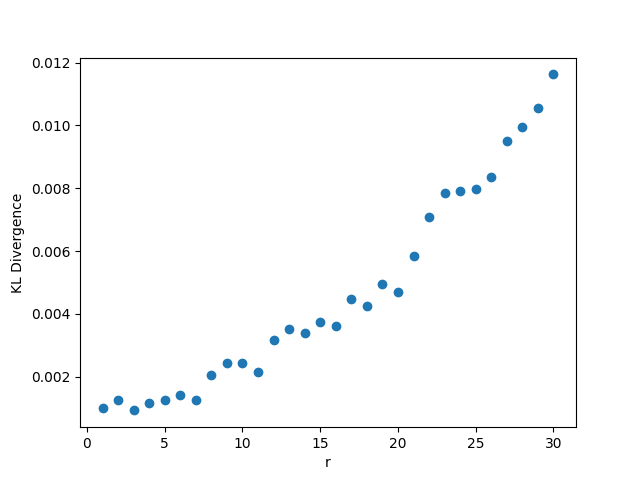}
\caption{Changing $r$ (range 1 to 30)}
\label{sfig:2c}
\end{subfigure}

\begin{subfigure}{0.33\textwidth}
\includegraphics[width=\linewidth]{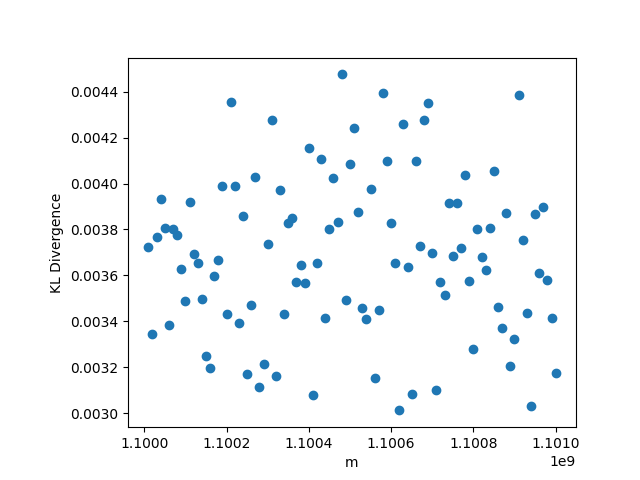}
\caption{Changing $m$ (range $1.1\times 10^9$ to $1.101\times 10^9$)}
\label{sfig:2d}
\end{subfigure}
\hspace{0.05\textwidth}
\begin{subfigure}{0.33\textwidth}
\includegraphics[width=\linewidth]{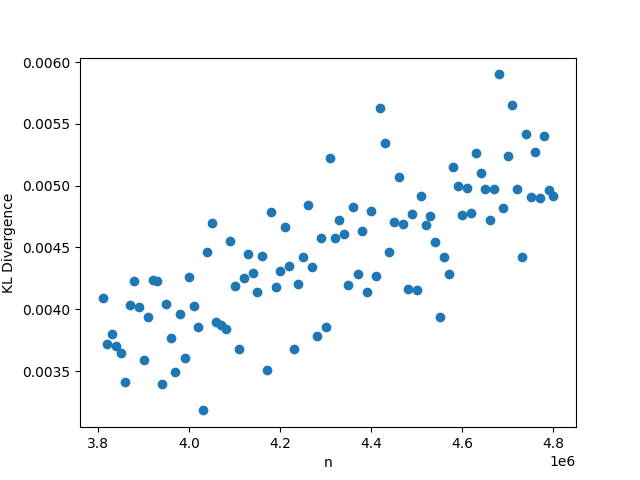}
\caption{Changing $n$ (range $3.8\times 10^6$ to $4.8\times 10^6$)}
\label{sfig:2e}
\end{subfigure}
\caption{Effect of changing parameters on the KL divergence comparing samples coming from our true distribution with those of the Normal distribution, in the setting that $m/n\rightarrow\infty$, $m\log(m) n^{-3/2}\rightarrow 0$. See Section~\ref{s:mf} for details.}
\label{fig2}
\end{figure}

\subsection{$\frac{m}{n}\rightarrow \alpha$}
\label{s:mc}
In this setting, whenever we fix $m$, we choose the value $10^6$, and $n=10^6$. As usual, when we fix the other parameters, we choose $r,s=15$ and $p=0.5$. The results may be seen in Figure~\ref{fig3}. Again in this setting, the KL divergence remains nearly 0 regardless of the values of the parameters, reinforcing Theorem~\ref{thm:main}.

\begin{figure}
\centering
\begin{subfigure}{0.32\textwidth}
\includegraphics[width=\linewidth]{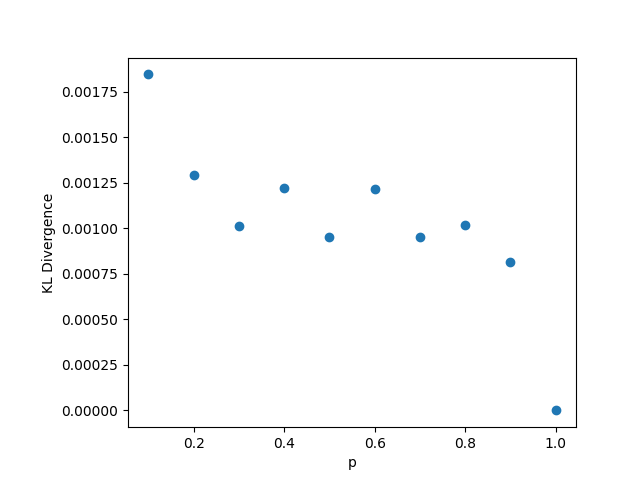}
\caption{Changing $p$ (range 0 to 1)}
\label{sfig:3a}
\end{subfigure}
\begin{subfigure}{0.32\textwidth}
\includegraphics[width=\linewidth]{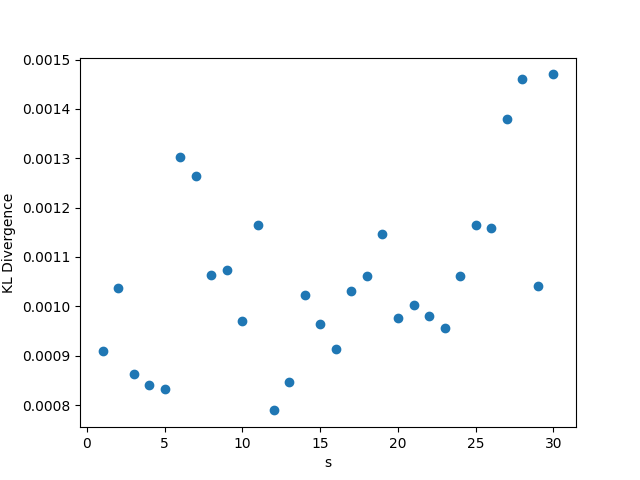}
\caption{Changing $s$ (range 1 to 30)}
\label{sfig:3b}
\end{subfigure}
\begin{subfigure}{0.32\textwidth}
\includegraphics[width=\linewidth]{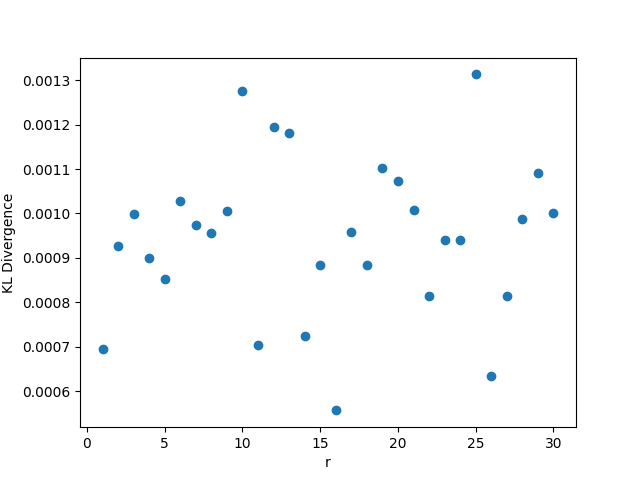}
\caption{Changing $r$ (range 1 to 30)}
\label{sfig:3c}
\end{subfigure}

\begin{subfigure}{0.33\textwidth}
\includegraphics[width=\linewidth]{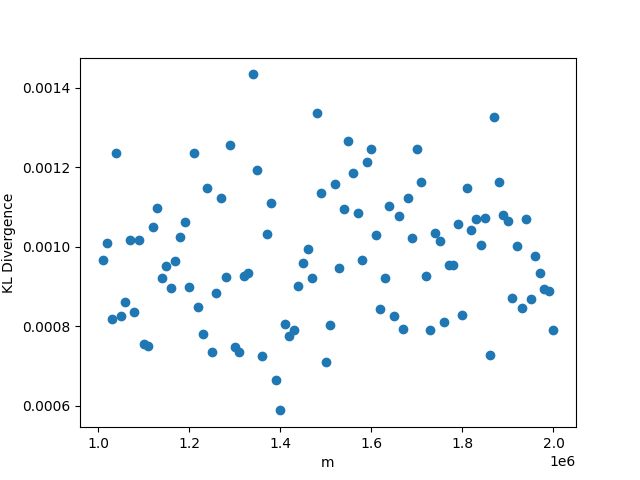}
\caption{Changing $m$ (range $10^6$ to $2\times 10^6$)}
\label{sfig:3d}
\end{subfigure}
\hspace{0.05\textwidth}
\begin{subfigure}{0.33\textwidth}
\includegraphics[width=\linewidth]{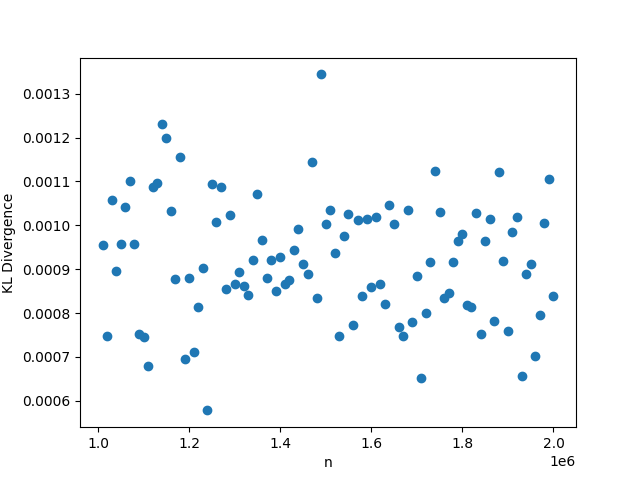}
\caption{Changing $n$ (range $10^6$ to $2\times 10^6$)}
\label{sfig:3e}
\end{subfigure}
\caption{Effect of changing parameters on the KL divergence comparing samples coming from our true distribution with those of the Normal distribution, in the setting where $m/n\rightarrow\alpha\in(0,+\infty)$. See Section~\ref{s:mc} for more details.}
\label{fig3}
\end{figure}

\subsection{$\frac{m}{n}\rightarrow 0$}
\label{s:m0}
In this setting, whenever we fix $m$, we choose the value $3.8\times 10^6$, and $n=1.1\times 10^9$. Unlike in previous cases, when the remaining parameters are fixed, we choose $r=15$ and $s=16$, while $p=0.5$. The results may be seen in Figure~\ref{fig4}. The reason for the change in the default value for $s$ can be explained by considering panels (b) and (c) where $r$ and $s$ vary. We can see that the KL divergence spikes when $r=s$: this comes from the fact that in this case, the hypothesized limiting distribution has 0 variance, so the Normal distribution we are comparing to collapses. Since for any finite $m$ and $n$, the distribution of $R$ is not 0, the KL divergence is much larger at this point. Excluding this case, the KL divergences in these plots are all small, indicating close similarity between the simulated distributions and the proposed hypothetical distributions. We also note in panel (d) that increasing $m$ results in a worse KL divergence, but this also means that the ratio $m/n$ is increasing, meaning that we are further from the limiting regime we are considering in this case. The maximum value of the KL divergence is still quite small for this range of $m$, however.

\begin{figure}
\centering
\begin{subfigure}{0.32\textwidth}
\includegraphics[width=\linewidth]{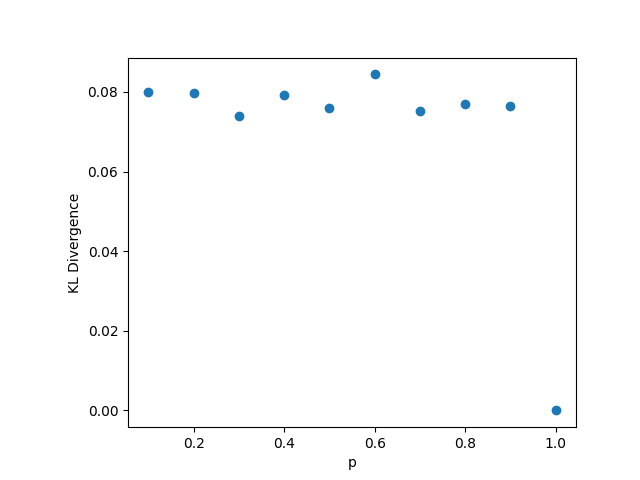}
\caption{Changing $p$ (range 0 to 1)}
\label{sfig:4a}
\end{subfigure}
\begin{subfigure}{0.32\textwidth}
\includegraphics[width=\linewidth]{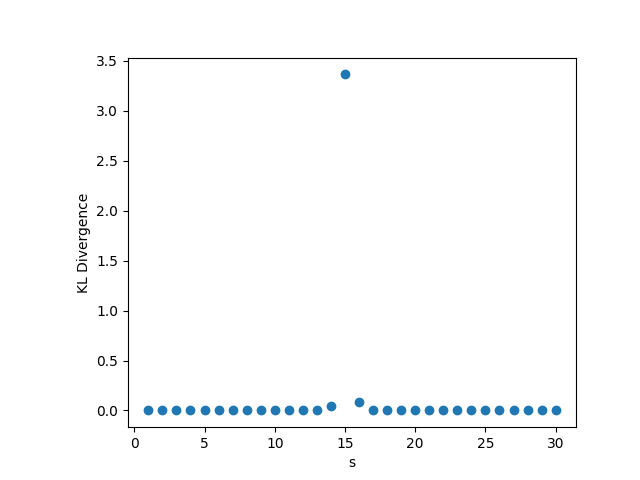}
\caption{Changing $s$ (range 1 to 30)}
\label{sfig:4b}
\end{subfigure}
\begin{subfigure}{0.32\textwidth}
\includegraphics[width=\linewidth]{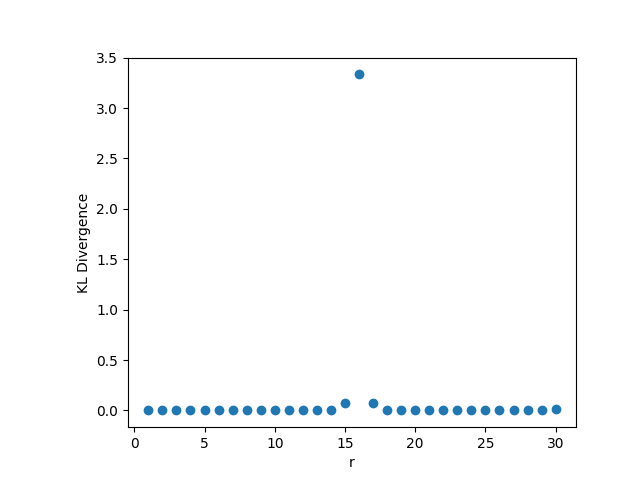}
\caption{Changing $r$ (range 1 to 30)}
\label{sfig:4c}
\end{subfigure}

\begin{subfigure}{0.33\textwidth}
\includegraphics[width=\linewidth]{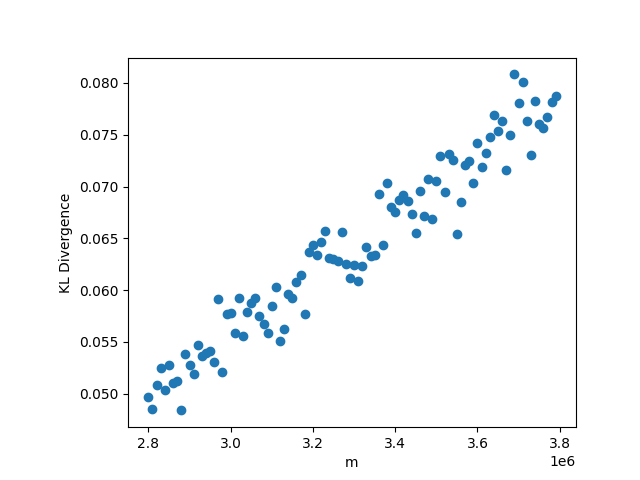}
\caption{Changing $m$ (range $2.8\times 10^6$ to $1.2\times 10^6$)}
\label{sfig:4d}
\end{subfigure}
\hspace{0.05\textwidth}
\begin{subfigure}{0.33\textwidth}
\includegraphics[width=\linewidth]{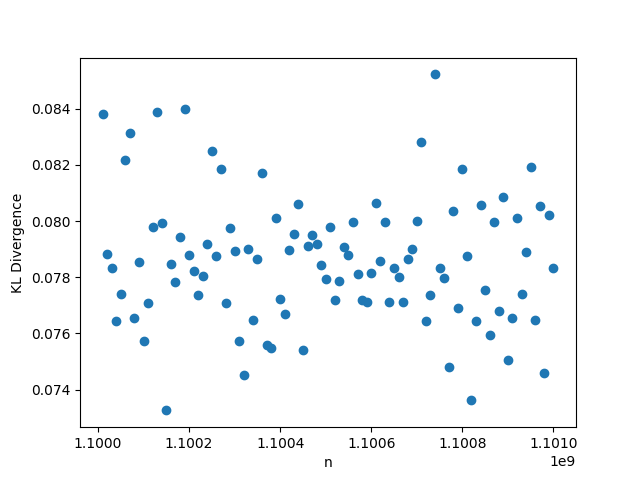}
\caption{Changing $n$ (range $1.1 \times 10^9$ to $1.101 \times 10^9$)}
\label{sfig:4e}
\end{subfigure}
\caption{Effect of changing parameters on the KL divergence comparing samples coming from our true distribution with those of the Normal distribution, in the setting where $m/n\rightarrow0$. Note that unlike in other sections, when we hold the parameter $s$ fixed, we choose the value $s=16$, since the hypothesized limiting Normal distribution has variance 0 when $r=s$ in this setting. See Section~\ref{s:m0} for more details.}
\label{fig4}
\end{figure}

\section{Conclusion}

We studied the limiting distribution of a function of two independent Binomial random variables, in the setting where the number of trials for both variables grows large, but the rates of growth for those two quantities may differ. We were able to show that under several parameter regimes, the limiting distribution after centering and scaling is Normal, with a given variance. However, this does not exhaust the range of possible values for which one could consider this function: For example, in \cite{lubberts2025}, the authors considered the case where $m\sim n^2$ (for the special case of $r=1, s=1/2$), which is not addressed by our results here. While our results determine the appropriate values for the mean and variance of $R$ in the case when $m$ and $n$ are large, determining the behavior of higher moments would require a more careful analysis of the quadratic remainder term than the one undertaken in the present work. As a final remark, we note that even with random variables as well-studied as the Binomial, there still remain many interesting questions to explore.

\bibliographystyle{plain}
\bibliography{references}

\end{document}